\documentclass[12pt,leqno]{amsart}
\usepackage[latin1]{inputenc}
\usepackage{amsmath,amsthm,amssymb}
\usepackage{bbold}
\usepackage{hyperref}
\usepackage{mathdots}
\usepackage{mathtools}
\usepackage{subcaption}
\usepackage{xcolor}
\usepackage{enumerate}
\newtheorem{theorem}{Theorem}
\newtheorem{lemma}[theorem]{Lemma}
\newtheorem{remark}[theorem]{Remark}

\newtheorem{definition}[theorem]{Definition}
\newtheorem{proposition}[theorem]{Proposition}

\newtheorem{conjecture}[theorem]{Conjecture}

\newcommand{\calN}{\mathcal{N}}

\newcommand{\calS}{\mathcal{S}}

\newcommand{\calV}{\mathcal{V}}

\newcommand{\bbC}{\mathbb{C}}

\newcommand{\bbN}{\mathbb{N}}

\newcommand{\bbR}{\mathbb{R}}

\newcommand{\rmc}{\mathrm{c}}

\newcommand{\set}[2]{\left\{ #1 \,\middle|\, #2 \right\}}
\newcommand{\abs}[1]{\left\lvert #1 \right\rvert}
\newcommand{\norm}[1]{\left\lVert #1 \right\rVert}

\newcommand{\dd}{\,\mathrm{d}}

\title{Cheeger's constant for the Gabor transform and ripples}
\author[R.~Alaifari, B.~Pineau, M.~A.~Taylor and M.~Wellershoff]{Rima Alaifari, Ben Pineau, Mitchell A.~Taylor and Matthias Wellershoff}
\date{\today}

\begin{document}

\begin{abstract}
    We discover a new instability mechanism for short-time Fourier transform  phase retrieval which yields that for any reasonable window function $\phi$ in any dimension $d$, the local stability constant $c(f)$ defined via
     \begin{equation*}
       \inf_{|\lambda|=1}\|f- \lambda g\|_{M^p(\mathbb{R}^{d})}\leq  c(f)\| |V_\phi f|-|V_{\phi} g|\|_\mathcal{D}, \hspace{5mm} \forall g\in M^p(\mathbb{R}^d),
    \end{equation*}
    is \emph{infinite} on a dense set of vectors for all  weighted fractional Sobolev norms $\mathcal{D}$,  up to the sharp maximal regularity level ensuring that the problem is well-defined.  This, in particular, answers an open problem of Rathmair (\emph{LSAA}~2024), who asked whether exponential concentration of the Gabor transform on $\mathbb{R}^2$ guaranteed a finite local stability constant.   For the specific case of Gabor phase retrieval, we further show that there is a complementary dense set where the local stability constant on $\mathbb{R}^{2d}$ is \emph{finite}. In particular, the mapping $f\mapsto c(f)$, sending a function to its local stability constant, is highly discontinuous.   

Our results extend and complement a series of fundamental stability theorems for Gabor phase retrieval which have been proven over the last ten years. Of particular note is the work of Grohs and Rathmair (\emph{Comm.~Pure Appl.~Math.}~2019 for $d=1$ and \emph{J.~Eur.~Math.~Soc.}~2022  for  $d\geq 1$), who showed that for sufficiently strong weighted Sobolev norms $\mathcal{D}$ on $\mathbb{R}^{2d}$, the local stability constant for Gabor phase retrieval is bounded by the inverse of the Cheeger constant of the flat metric conformally
multiplied by $|V_\phi f|$. As a consequence of our analysis, we determine two dense families of functions, one of which has associated Cheeger constant zero and the other strictly positive. We also revisit the stability problem for STFT phase retrieval on bounded subsets of the time-frequency plane, for more general windows, and for restricted signal classes, extending and simplifying many influential results in the literature.

  %for \emph{any} infinite dimensional subspace $E$ of $M^p(\mathbb{R}^d)$, if in the definition of $c(f)$ one only requires $g\in E$,  then $c(f)$ will \emph{still} be infinite for a dense collection of $f\in E$, independent of the choice of norm $\mathcal{D}$.%Although it is well-known that Gabor phase retrieval problem plays a fundamental role in coherent diffraction imagining and its solution unique up to natural identifications, our results for the first time give a definitive proof that local instabilities are omnipresent in all STFT phase retrieval problems, regardless of the promotions of regularity.
\end{abstract}

\maketitle

\section{Introduction}
Fourier phase retrieval is a fundamental problem in signal processing and computational imaging, where the goal is to reconstruct an image $f$ from measurements of its Fourier magnitude $|\mathcal{F}f|$. This problem appears frequently in crystallography, where one wishes to determine the structure of a specimen by sending a beam of $X$-rays at the object and measuring the corresponding diffraction pattern. Mathematically, if we call the image we want to observe $f$, then the detector records $|\mathcal{F}f|^2$, the squared modulus of the Fourier transform of $f$. Our task, then, is to recover $f$ from $|\mathcal{F}f|$, up to physically irrelevant ambiguities.\medskip

Since the Fourier phase retrieval problem is in general extremely ill-posed, practitioners have sought different experimental setups to reduce the ambiguities present when trying to reconstruct a function $f$ from the squared modulus of a measurement operator applied to $f$. In ptychography, a mask $\phi$ is placed in front of the sample, so that the detector now measures $|\mathcal{F}(f\phi)|^2$. The mask is then shifted and the experiment repeated, resulting in the family of measurements $|\mathcal{F}(fT_x\phi)|^2$, where $T_x$ denotes translation by $x$. The rationale for using this modified experimental setup is that it incorporates additional redundancy, which helps filter out possible ambiguities when attempting to recover the true signal $f$. Mathematically, this new setup provides a clear link between the phase retrieval problem and the field of time-frequency analysis, which is what we will exploit in our analysis below. 
\medskip

Given a window function $\phi\in L^2(\mathbb{R}^d)$, recall that the short-time Fourier transform $V_\phi:L^2(\mathbb{R}^d)\to L^2(\mathbb{R}^{2d})$ is defined by
\begin{equation}\label{STFT Def}
V_\phi f(x,\omega)=\mathcal{F}(T_x\overline{\phi}f) (\omega)=\int_{\mathbb{R}^d} f(t)\overline{\phi(t-x)}e^{-2\pi it\cdot \omega}dt, \ \ \text{for}\ x,\omega\in \mathbb{R}^d.
\end{equation}
It is easy to see that $\|V_\phi f\|_{L^2(\mathbb{R}^{2d})}=\|\phi\|_{L^2(\mathbb{R}^{d})} \|f\|_{L^2(\mathbb{R}^{d})}$. In particular, one can stably reconstruct $f$ from $V_\phi f$ for any non-trivial window function $\phi$. The fundamental problem in STFT phase retrieval is to stably reconstruct $f$ from the phaseless measurement $|V_\phi f|^2$. To make this precise, first observe that for any unimodular scalar $\lambda$, we have $|V_\phi f|^2=|V_\phi(\lambda f)|^2$. Thus, it is impossible to distinguish $f$ and $\lambda f$ from phaseless STFT measurements. Fortunately, this is a trivial obstacle to the recovery of $f$, as the global phase has no physical meaning. We therefore define the equivalence relation $f\sim g$ if $f=\lambda g$ for some unimodular scalar $\lambda$. The objective of STFT phase retrieval is to identify situations where this is the only ambiguity, that is, when one may recover the equivalence class $[f]$ from $|V_\phi f|$. In the special case that $\phi$ is the Gaussian, we refer to (\ref{STFT Def}) as the \emph{Gabor transform} $\mathcal{G} = V_\phi$ and to the recovery of $[f]$ from $|\mathcal{G} f|$ as Gabor phase retrieval.\medskip

At the algebraic level, there is a standard method for determining whether the operator $V_\phi$ performs phase recovery. Indeed, recall that the ambiguity function $\mathcal{A}f$  of a square-integrable signal $f$ is defined by $\mathcal{A}f(x,\omega):=e^{\pi ix\cdot \omega} V_ff(x,\omega)$. As is well known, $\mathcal{A}f$ is a continuous function that determines $f$ up to global phase. Moreover, we have the relation 
\begin{equation*}\label{Ambiguity to STFT}
    \mathcal{F}\left( |V_\phi f|^2\right)(\omega,-x)=\mathcal{A}f(x,\omega) \overline{\mathcal{A}\phi(x,\omega)}
\end{equation*}
between the ambiguity function and the phaseless short-time Fourier transform. From this identity, it is easy to see that whenever $\mathcal{A}\phi$ does not vanish too much, we may recover $\mathcal{A}f$ (and hence $[f]$) from $|V_\phi f|^2$. In particular, since the ambiguity function of the $d$-dimensional Gaussian is the $2d$-dimensional Gaussian and the ambiguity function of each of the Hermite functions only vanishes on circles, there are many examples of windows where STFT phase retrieval is possible on the whole of $L^2(\mathbb{R})$.
\medskip

Although the ambiguity function relation gives a method to reconstruct $f$ from $|V_\phi f|^2$, it is notoriously unstable, as it involves division by the ambiguity function of the window, which always decays at infinity and may possibly have many other zeros. To measure the local stability of recovering $f$ from $|V_\phi f|$ up to global phase, we make the following informal definition.
\begin{definition}\label{defn of stability}
    Fix a window function $\phi$, two norms $\mathcal{M}$ and $\mathcal{N}$ and a class of functions $E$ with the property that $\|h\|_\mathcal{N}, \||V_\phi h|\|_\mathcal{M}<\infty$ for all $h\in E$. We say that $V_\phi$ does $(\mathcal{M},\mathcal{N},E)$-local stable phase retrieval at $f\in E$ if there exists a constant $C(f)<\infty$ such that for all $g\in E$ we have
    \begin{equation*}
        \inf_{|\lambda|=1}\|f-\lambda g\|_\mathcal{N}\leq C(f) \| |V_\phi f|-|V_\phi g|\|_{\mathcal{M}}.
    \end{equation*}
    When $E$ is the collection of all functions for which the above quantities are finite, we will abbreviate and speak of $\mathcal{M}\to \mathcal{N}$ stability at $f$.
\end{definition}
To make Definition~\ref{defn of stability} precise, we need to specify the norms $\mathcal{M}$ and $\mathcal{N}$ as well as the admissible class of signals $E$ that $g$ ranges over. In view of the relation $\|V_\phi f\|_{L^2(\mathbb{R}^{2d})}=\|\phi\|_{L^2(\mathbb{R}^{d})} \|f\|_{L^2(\mathbb{R}^{d})}$, the most natural choice of norms would be $\mathcal{M}=L^2(\mathbb{R}^{2d})$ and $\mathcal{N}=L^2(\mathbb{R}^{d})$. However, there is a recent general result \cite{Alharbi2024Locality} that states that for any discrete or continuous frame on an infinite-dimensional Hilbert space, the set of functions with an infinite $L^2$-stability constant is dense. In particular, for any infinite-dimensional closed subspace $E\subseteq L^2(\mathbb{R}^d)$, the set of functions $f\in E$ with \emph{infinite} ($L^2(\mathbb{R}^{2d}$), $L^2(\mathbb{R}^{d})$, $E$)-local stability constant is dense in $E$.
\medskip

Knowing that there is no linear prior $E$ that may be imposed to avoid abundant instability in the $L^2$-norm, we will instead pursue a line of research pioneered by Grohs and Rathmair \cite{Grohs2019Stable,MR4404785}, who argued that local stability could be recovered if one uses weighted Sobolev norms for $\mathcal{M}$ and appropriate modulation norms for $\mathcal{N}$.
\subsection{History}
The study of the stability problem for phase retrieval in infinite-dimensions began with an influential work of Cahill-Casazza-Daubechies \cite{MR3554699} who showed that phase retrieval using a frame is always unstable. To make this precise, let $T:\mathcal{H} \to L^2(\mu)$ be an operator from a Hilbert space $\mathcal{H}$ into $L^2(\mu)$. We say that $T$ does \emph{stable phase retrieval} at $f\in \mathcal{H}$ if there exists a constant $C(f)<\infty$ so that for all $g\in \mathcal{H}$ we have
\begin{equation}\label{LSPR intro}
    \inf_{|\lambda|=1}\|f-\lambda g\|_{\mathcal{H}}\leq C(f)\| |Tf|-|Tg|\|_{L^2(\mu)}.
\end{equation}
The result of Cahill-Casazza-Daubechies \cite{MR3554699} states that whenever $T$ is the analysis operator of a discrete frame and $\dim \mathcal{H}=\infty$ then $\sup_{f\in \mathcal{H}}C(f)=\infty.$ This result was later extended by A.~and Grohs \cite{MR3656501}, who showed that whenever a discrete or continuous frame performs phase retrieval, the recovery map will always be continuous but not uniformly continuous. In fact, it was recently shown in \cite{Alharbi2024Locality} that for any discrete or continuous frame $T:\mathcal{H}\to L^2(\mu)$ defined on an infinite-dimensional Hilbert space we have $C(f)=\infty$ for a dense collection of $f\in \mathcal{H}$.
\medskip

After the above negative results, several distinct lines of research emerged. One major direction pioneered by Calderbank-Daubechies-Freeman-Freeman \cite{calderbank2022stable} has shown that it is possible to perform stable phase retrieval in infinite dimensions for many interesting operators $T$ that do not arise as the analysis operator of a frame. Most notably, through the combined efforts of Christ-P.-T.~\cite{MR4837558} and Freeman-Oikhberg-P.-T.~\cite{MR4800909}, infinite-dimensional subspaces $E$ of $L^2(\mathbb{R}^d)$ were constructed so that the operator $T:=\mathcal{F}|_E : E\to L^2(\mathbb{R}^d)$ satisfies \eqref{LSPR intro} with $\sup_{f\in E} C(f)<\infty.$ This led to a new and exciting line of research that utilized techniques from harmonic analysis, probability, and the geometry of Banach lattices to construct many relevant situations where phase information could be recovered in a stable manner.
\medskip

The second main line of research focused on gaining a deeper understanding of the inherent instabilities present in frame phase retrieval for specific structured frames such as the STFT or wavelet transforms. This has led to the formulation of stable phase retrieval in Definition~\ref{defn of stability}, where one now has the additional freedom to choose the norms $\mathcal{M}$ and $\mathcal{N}$, as well as the class of signals that $g$ ranges over. The first major result in this direction is due to A., Daubechies, Grohs and  Yin \cite{alaifari2019Stable}, who showed that Gabor phase retrieval restricted to bounded domains $\Omega\subseteq \mathbb{R}^2$ is $H^1$-stable. More specifically, when $\phi$ is the Gaussian, we have for any function $f\in L^2(\mathbb{R})$ and any generic bounded domain $\Omega\subseteq \mathbb{R}^2$ that there is a finite constant $C=C(f,\Omega)$ such that
\begin{equation}\label{LSPR bounded domains}
    \inf_{|\lambda|=1}\|\mathcal{G} f-\lambda \mathcal{G} g\|_{L^2(\Omega)}\leq C(f, \Omega)\| |\mathcal{G} f|-|\mathcal{G} g|\|_{H^1(\Omega)}\ \ \text{for all} \ g\in L^2(\mathbb{R}).
\end{equation}
 Note that in the language of Definition~\ref{defn of stability}, the above result corresponds to the case when $\mathcal{M}=H^1(\Omega)$ and $\mathcal{N}$ is the restricted modulation norm $\|f\|_{\mathcal{N}}=\|\mathcal{G} f\|_{L^2(\Omega)}$. Since functions in the image of the Gabor transform are (up to trivial factors) holomorphic, $\mathcal{N}$ is indeed a norm. Moreover, from \eqref{LSPR bounded domains} we are able to recover $\mathcal{G} f$ from $|\mathcal{G} f|$ on $\Omega$, which by analyticity uniquely recovers $f$. Note, however, that the inequality \eqref{LSPR bounded domains} does not contradict the above instability theorems for Gabor phase retrieval for three reasons. First, we are measuring in the $H^1$ norm rather than in the $L^2$ norm. Second, the instability of analytic continuation prevents us from \emph{stably} reconstructing $\mathcal{G} f$ from $\mathcal{G} f|_\Omega$ and third, the constant $C(f,\Omega)$ may decay to zero as $\Omega$ becomes larger.
\medskip

Inspired by the above results, Grohs and Rathmair published two celebrated articles \cite{Grohs2019Stable,MR4404785} showing that with a sufficient loss of regularity, the local stability constant for Gabor phase retrieval could be bounded by the inverse of the Cheeger constant of the flat metric conformally
multiplied by $|V_\phi f|$. More specifically, they derived an inequality of the form
\begin{equation}\label{GR}
    \inf_{|\lambda|=1}\|f-\lambda g\|_{\mathcal{M}^1}\lesssim \left(1+\frac{1}{h(f)}\right)\| |\mathcal{G} f|-|\mathcal{G} g|\|_{\mathcal{M}(f)}\ \ \text{for all} \ g\in\mathcal{M}^1.
\end{equation}
Here, $\mathcal{M}^1$ denotes the modulation norm based on $L^1$, $h(f)$ denotes an appropriate Cheeger constant adapted to $|\mathcal{G} f|$ and $\mathcal{M}(f)$ denotes a weighted first-order Sobolev norm that depends on $f$ in a way we will make precise below. The main message of \eqref{GR} is that by introducing more complicated norms and allowing for regularity loss, one could hope to overcome instabilities in Gabor phase retrieval. In particular, it was argued in \cite{Grohs2019Stable,MR4404785} that computing the Cheeger constant should be computationally tractable. In fact, we have the following conjecture of Rathmair which he presented during his lectures at the research term \emph{Lattice structures in analysis and applications} at ICMAT, Madrid, in 2024 (to access the slides and the conjecture, see \cite{Rathmairconjecture}).
\begin{conjecture}\label{Martin conjecture}
    Let $f\in L^2(\mathbb{R})$ be such that its Gabor transform $\mathcal{G}f$ is exponentially concentrated, i.e., there exists $\epsilon>0$ such that $|\mathcal{G}f|e^{\epsilon |z|}\in L^2(\mathbb{R}^2)$. Then the local stability and Cheeger constants of $f$ should be finite and non-zero.
\end{conjecture}
We remark that although the above conjecture suggests that ``most" well-localized signals should have a finite stability constant, an explicit estimation of $h(f)$ was only offered in \cite{Grohs2019Stable,MR4404785} when $f$ is the Gaussian function.
\medskip

To understand the importance of the inequality \eqref{GR}, we must recall the main sources of instability in phase retrieval. The phase retrieval problem (as well as various other important inverse problems, see \cite{abdalla2025sharp,camunez2025characterization,freeman2025optimal,garcia2025isometric, garcia2025existence}) can be defined in the general setting of Banach lattices, which also offers a complete and simple characterization of \emph{real} stable phase retrieval.
\medskip

Recall that a \emph{vector lattice} is a vector space $X$ equipped with a partial ordering $\leq$ that is compatible with the linear operations and for which $|f|:=f\vee (-f):=\sup\{f,-f\}$ exists for all $f\in X$. A \emph{Banach lattice} is a vector lattice that is also a Banach space and satisfies the compatibility condition
$$|f|\leq |g| \ \ \implies \ \ \ \|f\|\leq \|g\|.$$
It is easy to see that for any measure $\mu$ and compact Hausdorff space $K$, $L^p(\mu)$ and $C(K)$ are Banach lattices in the pointwise (a.e.) ordering. On the other hand, the Sobolev space $W^{1,p}$ is not a Banach lattice in this ordering. We say that a subspace $E$ of a Banach lattice $X$ does $C$-\emph{local stable phase retrieval (SPR)} at $f\in E$ if for all $g\in E$ we have
\begin{equation*}
    \inf_{|\lambda|=1}\|f-\lambda g\|\leq C\| |f|-|g|\|.
\end{equation*}
Note that the above definition generalizes the situation in \eqref{LSPR intro}, which is the special case $E:=T(\mathcal{H})\subseteq L^2(\mu)$. By making slight modifications to the proof of \cite[Proposition 3.4]{bilokopytov2025disjointly} one may establish the following theorem (note that $f \wedge g:= \inf\{f,g\}$).
\begin{theorem}\label{SPR Eugene}
    Let $E$ be a subspace of a real Banach lattice $X$ and $f\in E$. The following are equivalent for a constant $C<\infty$.
    \begin{enumerate}[1.]
        \item $C$-local stable phase retrieval holds at $f\in E$.
        \item There are no functions $g,h \in E$ with $f=g+h$ and $\| |g|\wedge |h|\|< \frac{1}{C}(\|g\|\wedge\|h\|)$.
    \end{enumerate}
\end{theorem}
 The above theorem shows that instabilities in real phase retrieval occur at $f$ if and only if $f$ can be partitioned into two pieces that are ``almost disjoint" from each other, i.e., the signal $f$ is almost disconnected. Returning to Gabor phase retrieval, the \emph{Cheeger constant} is an alternative measure of ``disconnectedness" of $\mathcal{G}f$, defined by
\begin{equation*}
    h(f):= \inf \frac{\| \mathcal{G}f\|_{L^1(\partial C)}}{\|\mathcal{G}f\|_{L^1(C)}},
\end{equation*}
where the infimum is taken over all open domains $C\subseteq \mathbb{R}^2$ with smooth boundary satisfying
\begin{equation*}
   \int_C|\mathcal{G}f|\leq \frac{1}{2}\int_{\mathbb{R}^2}|\mathcal{G}f|.
\end{equation*}
With the above discussion in mind, one may informally interpret \eqref{GR} as stating that the only source of instability in Gabor phase retrieval is ``disconnected measurements". Note that \eqref{GR} is not an immediate consequence of Theorem~\ref{SPR Eugene}, as Gabor phase retrieval is a \emph{complex} phase retrieval problem and Theorem~\ref{SPR Eugene} only applies in the real setting (i.e., to sign retrieval problems). Moreover, \eqref{GR} involves norms that are not Banach lattice norms and incurs a loss of regularity. Finally, we remark that although the Cheeger constant also quantifies ``disconnectedness" it does not immediately produce the functions $g,h$ in the image of the Gabor transform that would be needed to yield an instability in Theorem~\ref{SPR Eugene}.  In particular, although \eqref{GR} proves that the Cheeger constant is an upper bound for an appropriate local stability constant, it does not seem to claim that the Cheeger constant also bounds the local stability constant from below.
\subsection{Main results}
We now outline our main results. Our first result shows that no amount of imposed regularity or decay can overcome the inherent instabilities in STFT phase retrieval. In particular, $h(f)$ is infinite for a dense set of $f$ and Conjecture~\ref{Martin conjecture} is false.
\begin{theorem}[Dense instabilities for STFT phase retrieval]\label{Instabilities dense theorem}
 Fix $p,q\in [1,\infty)$, $0\leq s<1+\frac{1}{p}$, $r>0$ and a window function $\phi$. Suppose that $f\neq 0$ is such that 
 $\|V_\phi f\|_{W^{s',p}_{2r}\cap L^q(\mathbb{R}^{2d})}<\infty$ for some $s'>s$.
    Then for any $\epsilon>0$, there exists a function $f_\epsilon$ such that $\|V_\phi f-V_\phi f_\epsilon\|_{W^{s',p}_{r}\cap L^q(\mathbb{R}^{2d})}<\epsilon$ and $V_\phi$ fails to do $W^{s,p}_r\to L^q$ stable phase retrieval at $f_\epsilon$: There is no finite constant $C(f_\epsilon)$ such that 
\begin{equation}\label{theorem in}
    \inf_{|\lambda|=1}\|V_\phi f_\epsilon-\lambda V_\phi g\|_{L^q(\mathbb{R}^{2d})}\leq C(f_\epsilon)\| |V_\phi f_\epsilon|-|V_\phi g|\|_{W^{s,p}_r(\mathbb{R}^{2d})}
\end{equation} 
holds for all $g$ with $\|V_\phi g\|_{W^{s,p}_r\cap L^q(\mathbb{R}^{2d})}<\infty.$ 
\end{theorem}
%\begin{remark}
 %   Theorem~\ref{Instabilities dense theorem} is a qualitative version of what we will  prove below; our actual result will give a precise quantitative dependence on the parameters in  Theorem~\ref{Instabilities dense theorem} and incorporate integrability loss by involving $L^q$ norms for $q\neq p$ on the left-hand side of \eqref{theorem in}.
%\end{remark}

\begin{remark}
   The definition of stable phase retrieval requires a Lipschitz-type recovery. However, our proof will rule out any sort of local modulus of continuity of the recovery map.
\end{remark}
\begin{remark}
 
     For a general window function $\phi$, it is not clear what regularity functions in the image of $V_\phi$ have. Certainly, if $\phi\in L^2(\mathbb{R}^d)$ then $V_\phi: L^2(\mathbb{R}^d)\to L^2(\mathbb{R}^{2d})$, and one can construct a dense set of functions with infinite $L^2$-stability constant. Moreover, if $\phi$ is Schwartz, then the STFT $V_\phi$ maps Schwartz to Schwartz, so Theorem~\ref{Instabilities dense theorem} will produce a dense collection of functions that fail local stable phase retrieval even with regularity loss. Moreover, in the case of Gabor phase retrieval, our proof will construct a dense family of functions with zero Cheeger constant. On the other hand, if $\phi$ is rough, there is a priori no reason to expect any smoothness on $V_\phi f$, which is why we have phrased Theorem~\ref{Instabilities dense theorem} as saying that whenever $V_\phi f$ has some regularity, we can create an instability at that regularity level with vectors in the image of $V_\phi$ arbitrarily close to $V_\phi f$. Since we know very little about what functions lie in the image of $V_\phi$, this proof utilizes the symmetries of the STFT to produce an instability, and in particular we must choose $E$ to be invariant under the symmetries of the STFT. However, this is not completely necessary. For example, we will sketch how to prove that for any infinite-dimensional closed subspace $E\subseteq L^2(\mathbb{R})$ there is always a dense set of functions with infinite $(H^1(\mathbb{R}^2)$, $L^2(\mathbb{R}), E)$-stability constant for Gabor phase retrieval. In particular, neither a regularity loss nor \emph{any} priori restriction on the signal class $E$ (as long as it is an infinite-dimensional vector space) can remove instability in Gabor phase retrieval on $\mathbb{R}^2$.
\end{remark}
Theorem~\ref{Instabilities dense theorem} is consistent with the main results of Grohs and Rathmair in \cite{Grohs2019Stable,MR4404785}, as the latter results were \emph{conditional}: If $|\mathcal{G}f|$ is well-connected then $f$ does local stable phase retrieval in certain norms. However, to our knowledge, the only non-trivial example in the literature of a function with $h(f)\in (0,\infty)$ is the Gaussian, so it was unclear to what extent functions of the form $|\mathcal{G}f|$ are indeed ``well-connected". On the other hand, Theorem~\ref{SPR Eugene} and its variants have been successfully utilized outside of the realm of STFT phase retrieval to produce infinite-dimensional subspaces of Banach lattices doing stable phase retrieval with stability constant independent of $f$ \cite{calderbank2022stable, MR4837558, MR4800909}. That is, outside of the realm of STFT phase retrieval, one can actually construct real and complex subspaces where all normalized functions ``significantly overlap" and, in fact, $\sup_{f\in E}C(f)<\infty$.
\medskip

Our next main result states that for the specific case of Gabor phase retrieval, we in fact have \emph{stability} on a dense set as well!
\begin{theorem}
    There is a dense set of $f$ so that both the local stability constant for Gabor phase retrieval in the $H^1$-norm and the associated Cheeger constant of $f$ are finite and non-zero. More specifically, let $E=\{f\in L^2(\mathbb{R}) : \mathcal{G} f\in H^1(\mathbb{R}^2)\}.$ Then for a dense set of $f\in E$ there exists a constant $C(f)<\infty$ such that for all $g\in E$ we have
    \begin{equation*}
        \inf_{|\lambda|=1}\| f-\lambda g\|_{L^2(\mathbb{R})}\leq C(f) \||\mathcal{G}  f|-| \mathcal{G}  g|\|_{H^1(\mathbb{R}^2)}.
    \end{equation*}
\end{theorem}
\begin{remark}
   The stability of Gabor phase retrieval can also be partially propagated to other windows; we will explain this in more detail below.
\end{remark}
The mechanism we introduce to produce an instability for STFT phase retrieval is to attach to $f$ a sequence of ripples ``at infinity", while staying inside the image of the STFT. This will cause the local stability constant to blow up and hence the Cheeger constant to go to zero. We remark that, in some sense, this source of instability is a consequence of the mathematical model and  not the physics. More precisely, we note that the Cheeger constant defined above treats the whole time-frequency plane on an equal footing, whereas practitioners would only be able to take measurements on a compact subset of the time-frequency plane. In particular, instabilities ``at infinity" are not relevant in physical experiments. To conclude this article, we return to the question of whether the instabilities ``at infinity" are the \emph{only} source of instabilities, and show that this is indeed the case, even without regularity loss. More specifically, we will revisit the inequality \eqref{LSPR bounded domains} which establishes $H^1(\Omega)$ to $L^2(\Omega)$ stability on bounded domains and introduce an additional idea that will allow us to remove the derivative loss. 

% Our results are in stark contrast with the following stability theorems: (i) As shown by Alaifari and Grohs (\emph{SIAM J.~Math.~Anal.}~2017), the solution to the Gabor phase retrieval problem (i.e., the case where $\phi$ is the Gaussian)  is unique up to natural identifications and the recovery map $|V_\phi f|\mapsto f/\sim$ is continuous in $L^2$.  (ii) As shown by Alaifari, Daubechies, Grohs and  Yin (\emph{Found.~Comput.~Math.}~2019), for Gabor phase retrieval restricted to bounded domains $\Omega\subseteq \mathbb{R}^2$, $c(f)$ is finite for all functions $f$, even when measured in comparatively weak norms $\mathcal{D}$. (iii) As shown by Grohs and Rathmair (\emph{Comm.~Pure Appl.~Math.}~2019 for $d=1$ and \emph{J.~Eur.~Math.~Soc.}~2022  for  $d\geq 1$), for sufficiently strong norms $\mathcal{D}$ on $\mathbb{R}^{2d}$, the local stability constant for Gabor phase retrieval is bounded by the inverse of the Cheeger constant of the flat metric conformally
%multiplied by $|V_\phi f|$, which was generally believed to be computationally tractable.

\section{Instability: A dense set with infinite local stability constant}
The goal of this section is to prove that instabilities in STFT phase retrieval are inevitable, even if one imposes a loss of regularity and decay. We will begin by showing that no amount of qualitative concentration assumptions on the signal will restore stability; we will later upgrade this result so that it involves both a loss of regularity and decay.
\medskip

We define weighted Sobolev norms on $(\mathbb{R}^n, d\mu)$ via 
\begin{equation}\label{def of norm}
    \|f\|_{W^{s,p}_r(\mathbb{R}^n,d\mu)}:=\| \langle x\rangle^r f\|_{L^p(\mathbb{R}^n, d\mu)}+\| \langle D_x\rangle^sf\|_{L^p(\mathbb{R}^n, d\mu)}.
\end{equation}
When the measure is the Lebesgue measure, we will omit it from the notation. We recall the non-trivial fact \cite{MR1111186,MR2883848, MR1950719, pineau2025optimal} that the map $f\mapsto |f|$ is bounded on $W^{s,p}(\mathbb{R}^n)$ (in the sense that $\| |f|\|_{W^{s,p}(\mathbb{R}^n)}\lesssim \|f\|_{W^{s,p}(\mathbb{R}^n)}$) if and only if $s<1+\frac{1}{p}$. This will serve as a natural regularity threshold below. 
\medskip

To state our first result, we fix a distinguished $\sigma\geq 0$. We write $X^p_{\sigma}\subseteq L^p_\sigma$ to denote a closed subspace of $L^p_{\sigma}$ which is closed under translation with the induced norm 
$$\|f\|_{X_{\sigma}^p}:=\|\langle x\rangle^\sigma f\|_{L^p(\mathbb{R}^d)}.$$
With this notation, we may state the following general instability result.
\begin{proposition}\label{Main prop}
    Let $\sigma\geq 0$ and let $1\leq p<\infty$ and $1\leq q\leq \infty$. Let $X_{\sigma}^p$ be a non-empty closed subspace of $L^p_{\sigma}$ as defined above. Fix a nonzero $h\in X_{\sigma}^p$ with the decay hypothesis $h\in L_{2\sigma}^p\cap L^q_\sigma$. Then for every $\delta>0$ there exists $k\in X_{\sigma}^p\cap L^q$ such that
    \begin{equation}\label{star}
        \|h-k\|_{X_{\sigma}^p\cap L^q}<\delta
    \end{equation}
    and the $X^p_\sigma\to L^q$ SPR constant for $k$ is infinite. More precisely, there is a sequence of functions $k_n\in X_{\sigma}^p\cap L^q$ and an increasing sequence of real numbers $a_n\to \infty$ such that $k_n\to k$ in $X^p_{\sigma}\cap L^q$ and for large enough $n$,
    \begin{equation}\label{starstar}
        \inf_{|\lambda|=1}\|k-\lambda k_n\|_{L^q}> a_n \| |k|-|k_n| \|_{X^p_\sigma}.
    \end{equation}
\end{proposition}
\begin{proof}\renewcommand{\qedsymbol}{}
    In the proof below, we will assume that $h$ has some mass at the origin, i.e., 
    \begin{equation}\label{normalize}
    \min( \|h\|_{L^p(B)}, \|h\|_{L^q(B)})=1, 
    \end{equation}
    where $B$ is the unit ball in $\mathbb{R}^d$. If this is not the case, one can adjust the proof about a point where $h$ has mass. We define a family of disjoint annuli $A_n:=\{j_n\leq |x|\leq 2j_n\}$ where $j_n\geq 1$ is an increasing sequence of integers to be chosen. We will adopt the notation that  
    \begin{equation*}
        \begin{split}
            \|f\|_{X(A_n)}\coloneqq\|\chi_{A_n}f\|_{X_{\sigma}^p}, \ n\in \mathbb{N},\hspace{5mm}\|f\|_{X_{\geq n}}\coloneqq\|\chi_{|x|\geq \frac{j_n}{2}}f\|_{X_{\sigma}^p},
        \end{split}
    \end{equation*}
    where $\chi_S$ denotes the indicator function of $S$. We will also use the notation $\|\cdot\|_{L^q(A_n)}$ and $\|\cdot\|_{L^q_{\geq n}}$ which are defined similarly to the above. We can then choose $j_n$ large enough so that the annuli $A_n$ are pairwise disjoint, and we have
    \begin{equation}\label{masstail}
        \|h\|_{X(A_n)\cap L^q(A_n)}\leq \|h\|_{X_{\geq n}\cap L_{\geq n}^q}\leq 2^{-3n}\langle j_n\rangle^{-\sigma}.
    \end{equation}
Above, by $\langle j_n\rangle^\sigma$ we mean $\langle(j_n,0,\dots ,0)\rangle^\sigma$. Note that \eqref{masstail} ensures that most of the $X^p_{\sigma}\cap L^q$ mass of $h$ is in the ball $B(0,\frac{j_n}{2})$. The reason we get the factor of $\langle j_n\rangle^{-\sigma}$ in the upper bound \eqref{masstail} is thanks to the hypothesis that $h\in X^p_{2\sigma}\cap L^q_\sigma$. The factor of $2^{-3n}$ comes from taking $j_n$ large enough to ensure that
\begin{equation*}
\|h\chi_{|x|\geq\frac{j_n}{2}}\|_{X_{2\sigma}^{p}\cap L_{\sigma}^q}\leq 2^{-3n}.    
\end{equation*}
Hence, in particular, $j_n$ may possibly be very large relative to $2^{3n}$.
\end{proof}

We now define a sequence of functions $\epsilon_n$ by
\begin{equation*}
    \epsilon_n=\frac{2^{-n}}{\langle j_n\rangle^\sigma}\tau_{\frac{3}{2}j_n}h
\end{equation*}
where $\tau_n$ denotes translation by $n$ in the $x_1$ direction. The main bounds for $\epsilon_n$ that we will need to prove instability are given by the following lemma.
\begin{lemma}\label{Lemma main}
    By taking $j_n$ growing sufficiently fast, the sequence $\epsilon_n$ satisfies the bounds 
    \begin{enumerate}
        \item (Global upper bounds). We have
        \begin{equation}\label{gub}
            \|\epsilon_n\|_{X_{\sigma}^p\cap L^q}\lesssim_\sigma 2^{-n} \|h\|_{X_{\sigma}^p\cap L^q},\hspace{5mm}\|\epsilon_n\|_{L^q\cap L^p} = 2^{-n} \langle j_n\rangle^{-\sigma}\|h\|_{ L^q\cap L^p}.
        \end{equation}
        \item (Mass concentration bound on $A_n$). We have
        \begin{equation}\label{mcb}
            \|\epsilon_n\|_{L^q(A_n)} \geq 2^{-n}\langle j_n\rangle^{-\sigma}.
        \end{equation}
        \item (Mass tail bound outside of $A_n$). Let $A_n^c$ denote the complement of $A_n$. Then
            \begin{equation}\label{mtb}
            \|\epsilon_n\|_{X(A_n^c)} \lesssim_\sigma 2^{-4n}\langle j_n\rangle^{-\sigma}.
        \end{equation}
        \item (Small overlap bound). We have
        \begin{equation}\label{sob}
            \sup_{1\leq j<n}\|\epsilon_j\|_{X(A_n)} \lesssim_\sigma 2^{-3n}\langle j_n\rangle^{-\sigma}.
        \end{equation}
In the above, the implicit constants  are universal (i.e., they do not depend on $n$, $h$, or $\epsilon_n$).
    \end{enumerate}
\end{lemma}
\begin{remark}
    The above lemma essentially says that the $\epsilon_n$ are of size $2^{-n}$ in $X_{\sigma}^p\cap L^q$ and are almost disjoint (in the sense that all but a factor of $\langle j_n\rangle^{-\sigma}2^{-3n}$ of their $X$ mass lives in $A_n$). Moreover, the $\epsilon_n$ are almost disjoint from $h$.
\end{remark}
We postpone the proof of this lemma until the end and will first show how Proposition \ref{Main prop} follows. 
\begin{proof}[Proof that  Lemma~\ref{Lemma main} implies Proposition~\ref{Main prop}]
Fix $0<\delta<1$. We will show that we can take $a_n=2^n$ in \eqref{starstar} (but this is not essential). It will be convenient to define $k:= c_n+b_n$ and $k_n:= c_n-b_n$ where $c_n$ and $b_n$ are sequences of functions given by
\begin{equation*}
    c_n=h+\delta\sum_{j\leq n}\epsilon_j
\end{equation*}
and 
\begin{equation*}
    b_n=\delta\sum_{j>n}\epsilon_j.
\end{equation*}
If $\delta$ is small enough, then by \eqref{gub} and geometric summing, we have
\begin{equation*}
\|h-k\|_{X^p_{\sigma}\cap L^q}\lesssim_{\|h\|_{X_{\sigma}^p\cap L^q}} \delta.    
\end{equation*}
    Moreover, it is clear that $k_n\to k$ in $X^{p}_{\sigma}\cap L^q$ as $n\to\infty$. It remains to establish \eqref{starstar}. To this end, we let $\lambda\in\mathbb{C}$ with $|\lambda|=1$. We first claim that for large enough $n$, we have the uniform in $\lambda$ bound 
    \begin{equation}\label{lowerboundlhs}
    \begin{split}
       \|k-\lambda k_n\|_{L^q}&=\|(1-\lambda)c_n+(1+\lambda)b_n\|_{L^q} \gtrsim_\sigma \delta 2^{-n}\langle j_{n+1}\rangle^{-\sigma}.
    \end{split}
    \end{equation}    
To see this, we consider separately the cases where $|\lambda-1|\geq\frac{1}{2}$ and $|\lambda-1|<\frac{1}{2}$. In the first case, we have for $\delta$ small enough,
\begin{equation*}
\|k-\lambda k_n\|_{L^q}\geq \frac{1}{2}\|h\|_{L^q}-2\delta\sum_{j\geq 1}\|\epsilon_j\|_{L^q}\gtrsim \|h\|_{L^q},    
\end{equation*}
where we used the triangle inequality and \eqref{gub}. This more than suffices for large $n$. If $|\lambda-1|<\frac{1}{2}$, we have
\begin{equation*}
\|k-\lambda k_n\|_{L^q(\mathbb{R}^d)}\geq \|k-\lambda k_n\|_{L^q(A_{n+1})}\geq \delta\|\epsilon_{n+1}\|_{L^q(A_{n+1})}-2\|h\|_{L^q(A_{n+1})}-2\delta\sum_{j\neq n+1}\|\epsilon_j\|_{L^q(A_{n+1})}  
\end{equation*}
which from \eqref{masstail}, \eqref{mcb}, \eqref{mtb} and \eqref{sob} can be bounded from below by $\delta 2^{-n}\langle j_{n+1}\rangle^{-\sigma}$ (if $n$ is large enough). This gives \eqref{lowerboundlhs}. On the other hand, we have
\begin{equation*}
    \| |k|-|k_n|\|_{X^p_{\sigma}}=\| |c_n+b_n|-|c_n-b_n| \|_{X_{\sigma}^p}\leq 2\|b_n\|_{X(\cap_{j>n}A_j^c)}+2\sum_{j>n}\|c_n\|_{X(A_j)}.
\end{equation*}
    By geometric summing $2^{-3n}$ and using \eqref{mtb}, it is easy to see that the first term satisfies
    \begin{equation*}
        \|b_n\|_{X(\cap_{j>n}A_j^c)}\lesssim_{\sigma}\langle j_{n+1}\rangle^{-\sigma}2^{-4n},
    \end{equation*}
    while the second term satisfies
    \begin{equation*}
        2\sum_{j>n}\|c_n\|_{X(A_j)}\leq 2\sum_{j>n}\|h\|_{X(A_j)}+2\sum_{j>n}\sum_{k\leq n}\|\epsilon_k\|_{X(A_j)}.
    \end{equation*}
    Using \eqref{masstail}, the first term in the above can be estimated by 
    \begin{equation*}
        2\sum_{j>n}\|h\|_{X(A_j)}\lesssim 2^{-3n}\langle j_{n+1}\rangle^{-\sigma}.
    \end{equation*}
    Using the small overlap bound \eqref{sob}, the second term is bounded by
    \begin{equation*}
        2\sum_{j>n}\sum_{k\leq n}\|\epsilon_k\|_{X(A_j)}\lesssim_{\sigma} \langle j_{n+1}\rangle^{-\sigma}\sum_{j>n}2^{-2j}\leq \langle j_{n+1}\rangle^{-\sigma}2^{-2n}.
    \end{equation*}
    By combining this bound with \eqref{lowerboundlhs}, we obtain \eqref{starstar}. This completes the proof.
\end{proof}
It remains to prove Lemma \ref{Lemma main}.
\begin{proof}
Below, we will use that $\langle x + y \rangle \lesssim \langle x \rangle \langle y \rangle$ for $x,y \in \mathbb R^d$ throughout the proof. We first observe that 
    \begin{equation*}
        \lVert \epsilon_n \rVert_{X_\sigma^p} = 2^{-n} \langle j_n \rangle^{-\sigma} \lVert \langle x + 3 j_n/2\rangle^\sigma \cdot h \lVert_{L^p} \lesssim_\sigma 2^{-n} \left( \frac{\langle 3 j_n / 2\rangle}{\langle j_n \rangle} \right)^\sigma \cdot \lVert h \rVert_{X^p_\sigma} \lesssim_\sigma 2^{-n} \lVert h \rVert_{X^p_\sigma}.
    \end{equation*}
    So, equation~\eqref{gub} follows from the translation invariance of the Lebesgue norms. We now move on to the bound \eqref{mcb}. For this, we simply note that by a translation change of variables and \eqref{normalize}, we have
\begin{equation*}
    \|\epsilon_n\|_{L^q(A_n)}=\frac{2^{-n}}{\langle j_n\rangle^\sigma} \|\tau_{\frac{3}{2}j_n}h\|_{L^q(A_n)}\geq \frac{2^{-n}}{\langle j_n\rangle^\sigma}\|h\|_{L^q(B)}\geq \frac{2^{-n}}{\langle j_n\rangle^\sigma},
\end{equation*}
where in the second last inequality we used that $A_n+\frac{3}{2}j_ne_1$ contains the unit ball. For equations~\eqref{mtb} and \eqref{sob}, we just reuse the argument above for equation~\eqref{gub}, the inclusions $A_n^\mathrm{c} - 3j_n/2 \subset B(0,j_n/2)^\mathrm{c}$ and $A_n - 3 j_\ell/2\subset B(0,j_n/2)^\mathrm{c}$ for $\ell < n$ (as well as equation~\eqref{masstail}).
\end{proof}

Fix a window function $\phi\neq 0$.
The fundamental \emph{covariance property} of the STFT states that  for all $x,u,\omega,\eta\in \mathbb{R}^d$ we have
\begin{equation}
    V_\phi(T_uM_\eta f)(x,\omega)=e^{-2\pi i u\cdot \omega} V_\phi f(x-u,\omega-\eta),
\end{equation}
where $T_u$ is translation by $u$ and $M_\eta$ is modulation by $\eta$. This gives us the invariance needed to apply Proposition~\ref{Main prop}. To prove Theorem~\ref{Instabilities dense theorem}, we now have to incorporate derivative loss.
%If the window function $\phi$ in the Schwartz class $\mathcal{S}$, the STFT $V_\phi$ will map Schwartz functions to Schwartz functions. In particular, there will be many functions $f$ for which $V_\phi f$ has high regularity. On the other hand, if the window $\phi$ has poor regularity, there may be few functions $f$ for which $V_\phi f$ has good smoothness properties. Therefore, in our instability result below, we will assume that $V_\phi f$ has a certain regularity, and given this, we will construct vectors which are arbitrarily close to $f$ and have infinite stability constant for phase retrieval at that regularity level. It is then an immediate consequence that whenever $\phi$ is Schwartz, the set of functions with infinite stability constant will be dense.

\begin{proof}[Proof of Theorem~\ref{Instabilities dense theorem}]
In view of Proposition~\ref{Main prop}, it remains to deal with the regularity loss, which we shall overcome by  using techniques from paradifferential calculus. Indeed, the argument below will allow us to reduce from the top order derivative norm to the lower order norm in a relatively simple way. 
\medskip

 Fix $V_\phi f$ as in Theorem~\ref{Instabilities dense theorem} and $\varepsilon>0$. Using the covariance property of the STFT and the construction in Proposition~\ref{Main prop}, we can create functions $V_\phi f_\epsilon$ as in Theorem~\ref{Instabilities dense theorem} as well as a sequence of functions  $V_\phi f_{\epsilon,k}$ which witness an instability at $V_\phi f_\epsilon$.  Morally speaking, $V_\phi f_\epsilon$ will take the form
    \begin{equation}\label{ansatz}
        V_\phi f_\epsilon=V_\phi f+\sum_{n=1}^\infty \epsilon_n T_{a_n}V_\phi f,
    \end{equation}
where $(\epsilon_n)$ is a sequence of small positive numbers, $(a_n)$ is a sequence of well-chosen positive numbers, and for $a>0$, $T_a V_\phi f(x,\omega)=V_\phi f(x,\omega_1-a,\omega_2,\dots,\omega_d)$ is the translation symmetry induced by modulation in the first coordinate. $V_\phi f_{\epsilon,k}$ is then obtained by flipping signs after the $k$-th bump.
\medskip

Fix $j$ and let $P_j$ denote the standard Littlewood-Paley projection at frequency $2^j$ (see \cite[Appendix A]{tao2006nonlinear} for the basic Littlewood-Paley theory we use below).
With the above notation, we write
\begin{equation*}
 |V_\phi f_{\epsilon,k}|- | V_\phi f_{\epsilon}|=P_{<j}\left(|V_\phi f_{\epsilon,k}|- |V_\phi f_{\epsilon}|\right)+P_{\geq j}\left(|V_\phi f_{\epsilon,k}|- | V_\phi f_{\epsilon}|\right).
\end{equation*}
From this we obtain the estimate
\begin{equation*}
     \||V_\phi f_{\epsilon,k}|- | V_\phi f_{\epsilon}|\|_{{W}^{s,p}}\lesssim 2^{js}\||V_\phi f_{\epsilon,k}|- | V_\phi f_{\epsilon}|\|_{L^{p}}+2^{-j\delta}\left(\||V_\phi f_{\epsilon,k}|\|_{W^{s+\delta,p}}+\||V_\phi f_{\epsilon}|\|_{W^{s+\delta,p}}\right),
\end{equation*}
for any $0<\delta\ll 1$. In the second term, we can kill the absolute value as long as $s+\delta<1+\frac{1}{p}$. By the triangle inequality and the absolute convergence of the sum in \eqref{ansatz} we then get
\begin{equation*}
   2^{-j\delta}\left(\||V_\phi f_{\epsilon,k}|\|_{W^{s+\delta,p}}+\||V_\phi f_{\epsilon}|\|_{W^{s+\delta,p}}\right)\lesssim 2^{-j\delta}\left(\|V_\phi f\|_{W^{s+\delta,p}_r}+\epsilon\right).
\end{equation*}
This can be made small by taking $j$ large, by our assumption that $\|V_\phi f\|_{W^{s',p}_r}<\infty$. Thus, it suffices to control the weighted lower order norm, which we proved in Proposition~\ref{Main prop} can be taken to rapidly converge to zero.% it suffices to show that  we can choose $(\epsilon_n)$ and $(a_n)$ satisfying  \eqref{Constraint} such that 
%\begin{equation}\label{Level of Lp}
  %  \|\langle x\rangle^N\left(|V_\phi f_{\epsilon,k}|- | V_\phi f_{\epsilon}|\right)\|_{L^{p}}\to 0.
%\end{equation}
\end{proof}

\begin{remark}
    In view of the above proof, it should be evident to the reader that similar instability results hold if the STFT is replaced by the wavelet transform.
\end{remark}

\begin{remark}
    Note that our definition of the norm \eqref{def of norm} only incorporates weights on the level of $L^p$. This is not crucial for the above arguments. For example, if we wanted to prove an instability with respect to the norm
     $$\|f\|_{W^{s,p}_r}+\sum_i\|\langle x\rangle^N\partial_{x_i} f\|_{L^p}$$
     instead of $\|f\|_{W^{s,p}_r}$ we would simply assume a higher regularity variant of the conditions in Theorem~\ref{Instabilities dense theorem}. Indeed, with these modifications we could still repeat much of the above argument, as differentiation is translation invariant. On the other hand, to prove the reduction from weighted Sobolev norms to weighted $L^p$ norms, we notice that, by chain rule, we have the norm equivalence
         $$\|f\|_{W^{s,p}_r}+\sum_i\|\langle x\rangle^N\partial_{x_i} f\|_{L^p} \sim_N\|f\|_{W^{s,p}_r}+\sum_i\|\partial_{x_i}\left(\langle x\rangle^N f\right)\|_{L^p}.$$
     If one uses this latter expression of the norm, the Littlewood-Paley argument goes through almost verbatim if one replaces $|V_\phi f_{\epsilon,k}|- | V_\phi f_{\epsilon}|$  by $\langle x\rangle^N\left( |V_\phi f_{\epsilon,k}|- | V_\phi f_{\epsilon}|\right)$ and makes the appropriate absolute convergence assumption for the series \eqref{ansatz} which would follow from higher regularity variants of the conditions in Theorem~\ref{Instabilities dense theorem}.
\end{remark}

So far, we have worked only with $W^{s,p}_r\to L^q$ stability, which is defined by quantifying over all $g$ such that  $V_\phi g\in W^{s,p}_r\cap L^q(\mathbb{R}^{2d})$. 
What if we a priori know that $V_\phi f \in E\subseteq  W^{s,p}_r\cap L^q(\mathbb{R}^{2d})$, and we only want to quantify over $g$ with the same property in our definition of stability at $f$?  This brings us back to the full generality provided by Definition~\ref{defn of stability}.
Since $E$ may not be closed under any sort of symmetry of the STFT, the above argument for producing an instability breaks down. However, since the STFT is a frame, we will still obtain $(L^2(\mathbb{R}^{d}),L^2(\mathbb{R}^{2d}),E)$ instability on a dense set  \cite{Alharbi2024Locality}. Moreover, in the case of the Gabor transform we may combine the argument in \cite{Alharbi2024Locality}    where the compactness of the analysis operator on finite measure sets  is used to produce an $L^2$-almost disjoint sequence with standard elliptic regularity theory to obtain an almost disjoint sequence with respect to a higher regularity (say $H^1$) norm. By the above argument, this  can then be used to create  instabilities in  higher regularity norms  (say $H^1\to L^2$) for the Gabor transform at $\mathcal{G}f_\epsilon\in E$ arbitrarily close to $\mathcal{G}f\in E$, while only needing to quantify over $\mathcal{G}g\in E$ in the definition of stability.  Thus, as long as $E$ is an \emph{infinite-dimensional closed subspace}, instabilities will be omnipresent.
\medskip

On the other hand, there are very interesting examples where $E$ is \emph{not} a subspace, and one can prove stability for frame phase retrieval. In particular, we refer the reader to \cite{MR3554699,freeman2025cahill} for examples of frames for infinite-dimensional Hilbert spaces which perform stable phase retrieval on a relevant class of nonlinear subsets. Such subsets are created by enforcing \emph{quantitative decay} assumptions on the admissible signals by imposing a Besov-style condition at each frequency scale, which prohibits one from sending bumps to infinity without being forced to significantly reduce their norm. This condition rules out the mechanism for instability we observed above and, in many cases, restores stability.

%\footnote{This will involve some RKHS argument or an argument that the analysis operator on compact sets when restricted to $E$ allows one to produce an almost disjoint sequence consisting of elements of $E$.}

\section{Stability: Bounded domains and polynomials}

% Therefore, although the results of Grohs and Rathmair classify, in a sense, all the local instabilities for Gabor phase retrieval, our results  show that such instabilities are nevertheless omnipresent, even in the extreme case when measurements are taken in high regularity weighted Sobolev norms and the signals only need to be recovered in $L^p$. This is the first result of its kind for STFT phase retrieval problems; moreover, our results also apply to the wavelet transform and are equally new in that setting.

% For the specific case of Gabor phase retrieval, we further show that for \emph{any} infinite dimensional subspace $E$ of $M^p(\mathbb{R}^d)$, if in the definition of $c(f)$ one only requires $g\in E$,  then $c(f)$ will \emph{still} be infinite for a dense collection of $f\in E$, independent of the choice of norm $\mathcal{D}$.
We now return to the problem of establishing \emph{stability} for Gabor phase retrieval. From the above results, we know that instabilities are omnipresent on $\mathbb{R}^{2d}$. Nevertheless, we will prove that there are also many functions that  allow for stable recovery on $\mathbb{R}^{2d}$ and that \emph{all} functions allow for stable recovery on bounded subsets of the time-frequency plane.
\medskip

The results in this section are mostly specific to the Gabor transform (i.e., the Gaussian window) but we will also discuss possible extensions to other windows.

\subsection{Bounded domains}

%Since the Cheeger constant was generally believed to be  computationally tractable,  the above connections suggested that local stability in sufficiently strong norms on $\mathbb{R}^{2d}$ for Gabor phase retrieval  should  be achievable. However, our results imply that the Cheeger constant for Gabor phase retrieval is  zero on a dense set, and, moreover, our proof shows that this is also true for general windows $\phi$. 

%Moreover, we do so without derivative loss both on $\mathbb{R}^{2d}$ and bounded domains, improving the aforementioned results of Alaifari, Daubechies, Grohs and  Yin.
%We also explain how to remove derivative loss from previous results and show how to propagate certain stability results for the Gabor transform to a more general class of windows.
For a generic bounded domain $\Omega\subseteq \mathbb{C}$, we now give a simple proof of $H^1(\Omega)\to H^1(\Omega)$ local stability for all $f$. In the next subsection, we will prove a similar result for a dense set of $f$ when $\Omega=\mathbb{C}$. For this purpose, it is convenient to work directly in the Fock space, i.e., view the Gaussian as part of the measure rather than as part of the functions of interest.
\medskip

We start with a standard and elementary argument taken from \cite{alaifari2019Stable,Grohs2019Stable,MR4404785}. The key fact about holomorphic functions that we will use is the identity
\begin{equation}\label{key identity}
    |\nabla |F||=\frac{1}{\sqrt{2}}|\nabla F|=|F'|.
\end{equation}
\begin{remark}
    Note that one can always bound the gradient of the moduli by the moduli of the gradient, but this is not useful in phase retrieval. The converse inequality is true for real-valued functions and for holomorphic functions. It is very useful, as it  allows one to introduce moduli into the argument.

\end{remark}

The first elementary step of the argument is to note that for any domain $\Omega\subseteq \mathbb{C}$ (bounded or unbounded) we have
\begin{equation*}
    \inf_{|\lambda|=1}\|F_2-\lambda F_1\|_{L^p(\Omega)}\lesssim \inf_{\lambda\in \mathbb{C}}\|F_2-\lambda F_1\|_{L^p(\Omega)}+\||F_1|-|F_2|\|_{L^p(\Omega)}.
\end{equation*}
This inequality involves no special properties of the functions. We now write (where $d\gamma$ is the Gaussian measure)
\begin{align*}
   % \begin{split}
        \inf_{\lambda\in \mathbb{C}}\|F_2-\lambda F_1\|_{L^p(\Omega)}=\|\frac{F_2}{F_1}-\lambda \|_{L^p(\Omega, |F_1|^pd\gamma)}&\leq C_{poinc}(\Omega,|F_1|^pd\gamma)\|\nabla\left(\frac{F_2}{F_1}-\lambda \right)\|_{L^p(\Omega, |F_1|^pd\gamma)}
        \\
        &=C_{poinc}(\Omega,|F_1|^pd\gamma)\|\nabla\left(\frac{F_2}{F_1}\right)\|_{L^p(\Omega, |F_1|^pd\gamma)},
   % \end{split}
\end{align*}
i.e., we remove the phase by passing to the level of derivatives. We are now able to apply \eqref{key identity} to bound the above by something only involving moduli. Indeed, inserting \eqref{key identity} and using the quotient rule and the triangle inequality, we may bound
\begin{align*}
   % \begin{split}
        &\inf_{|\lambda|=1}\|F_2-\lambda F_1\|_{L^p(\Omega)} \nonumber \\
        &\lesssim C_{poinc}(\Omega,|F_1|^pd\gamma)\left(\||F_1|-|F_2|\|_{L^p(\Omega)}+\|\nabla|F_1|-\nabla|F_2|\|_{L^p(\Omega)}  +\big\|\frac{\nabla |F_1|}{|F_1|}\left(|F_1|-|F_2|\right)\big\|_{L^p(\Omega)}\right). \label{elem stab estimate}
  %  \end{align}
\end{align*}
The first two terms above are harmless. The last term has a bad balance of derivatives and needs to be handled in a case-by-case manner. Note first that, for the case of a Gaussian, i.e., $F_1=$ constant, this term is completely harmless.
\medskip

From the above, there are three things left to do if one wishes to prove stability for Gabor phase retrieval on $\Omega$:
\begin{enumerate}
    \item Justify that $C_{poinc}(\Omega,|F_1|^pd\gamma)<\infty$.
    \item Handle the third term  above, i.e., control $\big\|\frac{\nabla |F_1|}{|F_1|}\left(|F_1|-|F_2|\right)\big\|_{L^p(\Omega)}$.
    \item Minimize the derivative loss, e.g., replace $\inf_{|\lambda|=1}\|F_2-\lambda F_1\|_{L^p(\Omega)}$ by $\inf_{|\lambda|=1}\|F_2-\lambda F_1\|_{W^{1,p}(\Omega)}$.  %by a higher order norm or the RHS by a lower order norm\footnote{Replacing the RHS by a lower order norm won't be easy, but below we will justify making the LHS stronger.}.
\end{enumerate}
Let us first consider the case studied in \cite{alaifari2019Stable} where we have a bounded domain $\Omega$ and we assume that $F_1$ has $n$-zeros on $\Omega$ and none on the boundary. If $n=0$, then obviously $|F_1|^pd\gamma\sim d\gamma$, so the Poincar\'e inequality is justified, and to handle the third term above, we can just place $\frac{\nabla |F_1|}{|F_1|}$ in $L^\infty.$ By a simple inductive strategy, it now suffices to consider the case of one zero, and show that that zero is harmless.
Let this zero be at $z_0$ and consider a small ball $B_\epsilon(z_0)$. We have 
\begin{equation}\label{zeros dont matter}
    \inf_{|\lambda|=1}\|F_2-\lambda F_1\|_{L^p(\Omega)}\leq  \inf_{|\lambda|=1}\left(\|F_2-\lambda F_1\|_{L^p(B_{\epsilon}(z_0))}+\|F_2-\lambda F_1\|_{L^p(\Omega\setminus B_{\epsilon}(z_0))}\right).
\end{equation}
For the first term, we use the following standard proposition.
\begin{proposition}
    Suppose that $D\subseteq \mathbb{C}$ is simply connected and $F$ is holomorphic and bounded on a neighborhood of $D$. Then 
    $$\|F\|_{L^p(D)}\lesssim \|F\|_{L^p(\partial D)}.$$
\end{proposition}
This proposition allows us to place the first term in \eqref{zeros dont matter} on the boundary of $B_{\epsilon}(z_0)$. Then, one has several options. One option is to go to $W^{1,p}$ on a small  annulus with inner radius $\partial B_{\epsilon}(z_0)$ by the trace inequality.  This was done in previous works, but it is quite wasteful, as it neglects the fact that our functions are holomorphic. Instead, it is better to apply a Cacciopoli argument to not lose regularity. From this one deduces that % \footnote{This basically says that for harmonic functions, the, say, $W^{1,2}$ norm on a ball of radius $R$ is bounded from above by the $L^2$ norm on a ball of radius $R+\epsilon$, i.e., by expanding the domain, you can replace derivative norms with $L^2$ norms. Here, we are in a situation where expanding the annulus we got from trace theorem slightly to remove the derivative loss is allowed (this is why there is an $\epsilon/2$ radius in \eqref{zdm}). Note that such an argument wouldn't be possible if the zero were on $\partial\Omega$, as there would be no room to average.} 
\begin{equation}\label{zdm}
    \inf_{|\lambda|=1}\|F_2-\lambda F_1\|_{L^p(\Omega)}\lesssim \inf_{|\lambda|=1} \|F_2-\lambda F_1\|_{L^p(\Omega\setminus B_{\frac{\epsilon}{2}}(z_0))},
\end{equation}
which reduces to the case of no zeros. One can similarly use this to upgrade regularity on the left-hand side of the inequality. Indeed, one may write
\begin{equation*}
    \|\nabla F_2-\lambda \nabla F_1\|_{L^p(\Omega)}=  \| F_2'-\lambda  F_1'\|_{L^p(\Omega)}\lesssim\|F_2'-\lambda F_1'\|_{L^p(\Omega\setminus B_{\frac{\epsilon}{2}}(z_0))}\approx\|\nabla F_2-\lambda \nabla F_1\|_{L^p(\Omega\setminus B_{\frac{\epsilon}{2}}(z_0))}
\end{equation*} to reduce to the case where $F_1$ does not vanish. Then, one may write
\begin{equation*}
    \|\nabla F_2-\lambda \nabla F_1\|_{L^p(\Omega\setminus B_{\frac{\epsilon}{2}}(z_0))}=\|\nabla\left(\left(\frac{F_2}{F_1}-\lambda\right)F_1\right)\|_{L^p(\Omega\setminus B_{\frac{\epsilon}{2}}(z_0))}.
\end{equation*}
If the derivative falls on the first term, we reduce to the case studied above. Otherwise, the derivative falls on the second term, in which case we may write
$$\left(\frac{F_2}{F_1}-\lambda\right)\nabla F_1=\left(\frac{F_2}{F_1}-\lambda\right)\frac{\nabla F_1}{F_1}F_1$$
and place $\frac{\nabla F_1}{F_1}$ in $L^\infty$ as before. This now removes the derivatives and allows us to repeat the above argument, justifying objective (3) above.%, showing that additional the derivative. This argument  shows why zeros don't matter at all for stability of Gabor phase retrieval.
\begin{remark}
    It is instructive to compare the above argument with the previous approaches from \cite{alaifari2019Stable,Grohs2019Stable,MR4404785}. As emphasized in \cite{Grohs2019Stable}, in the case where $\mathcal{G}f$ is a Gaussian and $\Omega=B_R(0)$ is the ball centered at zero of radius $R$, the stability bounds in \cite{alaifari2019Stable} degenerate like $e^{-\pi R^2}$. This is because in \cite{alaifari2019Stable} the authors do not ``cancel" the Gaussian in $\frac{F_1}{F_2}$ and instead work with  $\frac{\mathcal{G}f_1}{\mathcal{G}f_2}$ when executing the above argument. They then individually estimate the numerator and denominator (from above and below, respectively) causing an exponential loss in $R^2$. This is rectified somewhat in \cite{Grohs2019Stable,MR4404785}, though these papers also treat the problematic term $\big\|\frac{\nabla |F_1|}{|F_1|}\left(|F_1|-|F_2|\right)\big\|_{L^p(\Omega)}$ quite differently. Indeed, in our analysis, we remove the zeros of $F_1$ and simply place $\frac{\nabla |F_1|}{|F_1|}$ in $L^\infty.$ In particular, in the Gaussian case where $F_1=1$, this does not cause any losses in $R$ (in fact, the problematic term completely vanishes). On the other hand, in \cite{Grohs2019Stable,MR4404785} the authors try to directly estimate  for a general function $H$,

    $$\big\| \frac{\nabla |F_1|}{|F_1|}H\big\|_{L^p(\Omega)}.$$
   Here, it is instructive to consider the case when $F_1$ is a polynomial, as below we will show that such functions allow for stable phase recovery on $\mathbb{C}$. In this case, we write $F_1(z)=(z-z_1)\cdots (z-z_n)$, where the $z_i$ are the zeros of $F_1$. Each of these zeros leads to a pole in $\left|\frac{\nabla |F_1|}{|F_1|}\right|=\left|\frac{\nabla F_1}{F_1}\right|$, and the question is how to estimate the contributions of these poles. In \cite[Proposition 4.3]{MR4404785}, the authors  ``quantify" the contribution of the poles by utilizing a weighted norm $\mathcal{M}(f)$ as in \eqref{GR}, resulting in an estimate of the form

  $$\big\| \frac{\nabla |\mathcal{G}f|}{|\mathcal{G}f|}H\big\|_{L^p(\Omega)}\lesssim \|(1+|\cdot-z_0|^{2d+2})H\|_{L^q(\Omega)},$$
  where $z_0$ is the maximum of $|\mathcal{G}f|$ \emph{on all of $\mathbb{C}$} and the indices are chosen appropriately (in particular, $p=2$ is \emph{not} allowed and contrary to our argument a significant loss of integrability occurs). Here, the main point is that the parameter $z_0$ depends on the maximum of $|\mathcal{G}f|$ not on $\Omega$, but on all of $\mathbb{C}$. In particular, when $\Omega=B_1(0)$ and the $(z_i)_{i=1}^n$ are chosen in $\Omega$, the maximum will likely be \emph{far} from $\Omega$, resulting in a large weight in the above norm. Similarly, when $\mathcal{G}f$ is a shifted Gaussian and $\Omega=B_1(0)$, the stability constant will  blow up (or more precisely the norm $\mathcal{M}(f)$ will grow as the authors hide the non-Cheeger portion of the stability constant in this norm) as the Gaussian shifts to infinity.

\end{remark}

\subsection{Polynomials on the complex plane}

We now move to the case of the whole space but with $F_1=p(z)$, a polynomial. By a gluing argument which we will present below and the above stability results on bounded domains, it suffices  to justify the above conditions for a half-plane that is sufficiently away from the ``bulk" of $p(z)$. Indeed, we partition the complex plane into a large disk or radius $R$, together with three half-planes so that all subsets significantly overlap, and such that $p(z)e^{-\pi|z|^2}$ is very small but non-vanishing outside of the main disk.
\medskip

To handle (2)  we can write $\left|\frac{\nabla F_1}{F_1}\right|=\left|\frac{p'}{p}\right|$, which can be chosen to be small on the half space as $p'$ is a lower order polynomial than $p$, so this term is harmless and can be placed in $L^\infty$. This also takes care of (3), when combined with the above arguments. For (1), we choose our half-planes for the polynomial to decay relative to the Gaussian measure so that we have log concavity, extend the density $V$ in the definition of log-concavity by $\infty$ (i.e., extend the measure by $0$) and use the fact that log-concave measures  on $\mathbb{R}^d$ have Poincar\'e inequalities. This allows us to deduce a Poincar\'e inequality on the half-space, and hence stability for Gabor phase retrieval on the half-space when combined with the arguments from the previous subsection.
%\medskip

%\begin{remark}
 %   With two competing dense sets, it is tempting to ask which set is ``bigger", or at least which occurs most in practice. 
%\end{remark}
%We remark that it is entirely unclear whether there are \emph{any} functions $f$ with a finite $L^2$-$L^2$ stability constant. It may be worth remarking that this problem also occurs in the sampled phase retrieval problem. More precisely, it is know that there is \emph{no} lattice $\Lambda$ in $\mathbb{R}^2$ with the property that, for all $f,g\in L^2(\mathbb{R})$, if $|\mathcal{G}f|$ and $|\mathcal{G}g|$ agree on $\Lambda$ then $[f]=[g]$. However, if one is also given the restrictions of the gradients of $|\mathcal{G}f|$ and $|\mathcal{G}g|$ to $\Lambda$ then one can very often deduce that $[f]=[g].$ In particular, gradient information is needed to make the sampling problem on a lattice well-posed.

%Questions on which dense set is bigger, or at least which occurs most in practice. Only reasonable things we could want is a Cheeger/stability constant that weights low frequency more than high frequency or priors that are nonlinear, e.g. concentration to fight the bad constant on bounded domains.

\subsubsection{Gluing argument}
In the above argument, we constructed open subsets of the complex plane which ``significantly overlap" and for which one has stability for Gabor phase retrieval at $\mathcal{G}^{-1}p$, where $p$ is a polynomial in the Fock space. We now systematically show that one may ``glue" these regions together in a way that maintains local stability. For this, we introduce for an open set $\Omega\subseteq \mathbb{C}$ and $r\geq 0$ the notation,
\begin{equation*}
    c_{r,\Omega}(f):=\sup_{g\in L^2(\mathbb{R})} \frac{   d_\Omega(f,g)}{\| |\mathcal{G}f|-|\mathcal{G}g|\|_{H^1_r(\Omega)}},
\end{equation*}
where 
\begin{equation*}
    d_\Omega(f,g):=\inf_{|\lambda|=1}\|\mathcal{G}f-\lambda\mathcal{G}g\|_{L^2(\Omega)}.
\end{equation*}
If $\Omega\subseteq \mathbb{C}$ is a domain and $A,B\subseteq \Omega$ are non-empty open subsets such that $\Omega=A\cup B$ we define the $f$-\emph{connectivity} of $A$ and $B$ as
\begin{equation*}
    \lambda_{A,B}(f):=\frac{\|\mathcal{G}f\|_{L^2(A\cap B)}}{\|\mathcal{G}f\|_{L^2(A)}+\|\mathcal{G}f\|_{L^2( B)}}>0.
\end{equation*}
Our gluing result may be stated as follows.
\begin{proposition}
    Let $\Omega\subseteq \mathbb{C}$ be a domain and let $A,B\subseteq \Omega$ be non-empty open subsets such that $A\cup B=\Omega$. Let $r\geq 0$ and fix $f\in L^2(\mathbb{R})$. Then
    \begin{equation*}
        c_{r,\Omega}(f)\leq \left(c_{r,A}(f)^2+c_{r,B}(f)^2\right)^\frac{1}{2}\left(\lambda_{A,B}(f)^{-1}+\sqrt{2}\right).
    \end{equation*}
\end{proposition}
\begin{proof}
    Let us denote the complex conjugates of the minimizers of $d_A(f,g)$ and $d_B(f,g)$ by $\tau_A$ and $\tau_B$, respectively. Then, we let
    \begin{equation*}
        \tau_0:=\text{avg}_{\mathbb{S}^1}(\tau_A,\tau_B):=\begin{cases}
    & \frac{\tau_A-\tau_B}{|\tau_A-\tau_B|},\ \text{if} \ \tau_A\neq -\tau_B,
    \\  
    &i\tau_A, \ \ \ \ \ \ \text{otherwise,}
    \end{cases}
    \end{equation*}
and note that 
\begin{equation*}
    | \tau_A-\tau_0|=|\tau_B-\tau_0|\leq 2^{-\frac{1}{2}} |\tau_A-\tau_B|.
\end{equation*}
We now estimate
\begin{equation*}
    \begin{split}
        d_\Omega (f,g)&\leq \|\tau_0 \mathcal{G}f-\mathcal{G}g\|_{L^2(\Omega)}
        \\
        &\leq   \|\tau_0 \mathcal{G}f-\mathcal{G}g\|_{L^2(A)}+ \|\tau_0 \mathcal{G}f-\mathcal{G}g\|_{L^2(B)}
        \\
        &\leq   \|\tau_A \mathcal{G}f-\mathcal{G}g\|_{L^2(A)}+ \|\tau_B \mathcal{G}f-\mathcal{G}g\|_{L^2(B)}+|\tau_A-\tau_0|\|\mathcal{G}f\|_{L^2(A)}+|\tau_B-\tau_0|\|\mathcal{G}f\|_{L^2(B)}
        \\
          &\leq   c_{A,r}(f)\| |\mathcal{G}f|-|\mathcal{G}g|\|_{H^1_r(A)}+ c_{B,r}(f)\| |\mathcal{G}f|-|\mathcal{G}g|\|_{H^1_r(B)}
          \\
          &+2^{-\frac{1}{2}}|\tau_A-\tau_B|(\|\mathcal{G}f\|_{L^2(A)}+\|\mathcal{G}f\|_{L^2(B)})
          \\
          &\leq \left(c_{A,r}(f)^2+c_{B,r}(f)^2\right)^{\frac{1}{2}}\left(\| |\mathcal{G}f|-|\mathcal{G}g|\|_{H^1_r(A)}^2+\| |\mathcal{G}f|-|\mathcal{G}g|\|_{H^1_r(B)}^2\right)^\frac{1}{2}
          \\
          &+2^{-\frac{1}{2}}\lambda_{A,B}(f)^{-1}|\tau_A-\tau_B|\|\mathcal{G}f\|_{L^2(A\cap B)}
          \\
          &\leq 2^\frac{1}{2}\left(c_{A,r}(f)^2+c_{B,r}(f)^2\right)^{\frac{1}{2}}\| |\mathcal{G}f|-|\mathcal{G}g|\|_{H^1_r(\Omega)}+2^{-\frac{1}{2}}\lambda_{A,B}(f)^{-1}\|(\tau_A-\tau_B)\mathcal{G}f\|_{L^2(A\cap B)}.
    \end{split}
\end{equation*}
We finally estimate the second summand by
\begin{equation*}
\begin{split}
    \|(\tau_A-\tau_B)\mathcal{G}f\|_{L^2(A\cap B)}&\leq   \|\tau_A\mathcal{G}f-\mathcal{G}g\|_{L^2(A\cap B)}+  \|\tau_B\mathcal{G}f-\mathcal{G}g\|_{L^2(A\cap B)}
    \\
    &\leq  \|\tau_A\mathcal{G}f-\mathcal{G}g\|_{L^2(A)}+  \|\tau_B\mathcal{G}f-\mathcal{G}g\|_{L^2(B)}
    \\
    &\leq 2^\frac{1}{2}\left(c_{A,r}(f)^2+c_{B,r}(f)^2\right)^\frac{1}{2} \| |\mathcal{G}f|-|\mathcal{G}g|\|_{H^1_r(\Omega)}.
    \end{split}
\end{equation*}
This completes the proof.
\end{proof}
\begin{remark}
    The stability of Gabor phase retrieval can be partially propagated to other windows. Indeed, in \cite{grochenig2020zeros} window functions were constructed for which the ambiguity function does not vanish and the stability of the STFT phase retrieval problem for these windows was studied in \cite{rathmair2024stable}. We remark that a quick explanation of why these windows should do STFT phase retrieval follows immediately from \eqref{Ambiguity to STFT} and the above stability argument for the Gabor phase retrieval problem. Indeed, for any window $\Phi$, we may use \eqref{Ambiguity to STFT}  to deduce that 
\begin{equation*}
\begin{split}
    \mathcal{F}\left( |V_\phi f|^2-|V_\phi g|^2\right)(\omega,-x)&=(\mathcal{A}f(x,\omega)-\mathcal{A}g(x,\omega)) \overline{\mathcal{A}\phi(x,\omega)}
    \\
    &=\overline{\frac{\mathcal{A}\phi(x,\omega)}{\mathcal{A}\Phi(x,\omega)}}(\mathcal{A}f(x,\omega) -\mathcal{A}g(x,\omega))\overline{\mathcal{A}\Phi(x,\omega)}
    \\
    &=\overline{\frac{\mathcal{A}\phi(x,\omega)}{\mathcal{A}\Phi(x,\omega)}}\mathcal{F}\left( |V_\Phi f|^2-|V_\Phi g|^2\right)(\omega,-x),
    \end{split}
\end{equation*}
where $\phi$ is the Gaussian window. Multiplying by $\langle (x,\omega)\rangle $ and applying Plancherel's theorem, one obtains

\begin{equation*}
    \| |V_\phi f|^2-|V_\phi g|^2\|_{H^1}=\|\langle (x,\omega)\rangle\overline{\frac{\mathcal{A}\phi(x,\omega)}{\mathcal{A}\Phi(x,\omega)}}\mathcal{F}\left( |V_\Phi f|^2-|V_\Phi g|^2\right)(\omega,-x)\|_{L^2}.
\end{equation*}
In \cite{grochenig2020zeros}, examples of $\Phi$ are constructed which satisfy the bound $\left|\langle (x,\omega)\rangle\frac{\mathcal{A}\phi(x,\omega)}{\mathcal{A}\Phi(x,\omega)}\right|\lesssim 1$. For such $\Phi$, we obtain 

\begin{equation}\label{No deriv}
    \| |V_\phi f|^2-|V_\phi g|^2\|_{H^1}\lesssim\||V_\Phi f|^2-|V_\Phi g|^2\|_{L^2}.
\end{equation}
Notably, only $L^2$ norms appear on the right-hand side of \eqref{No deriv}. One can now use our above stability results to -- in certain cases such as when $V_\phi f$ is a polynomial -- bound  $$\| |V_\phi f|^2-|V_\phi g|^2\|_{H^1}\ge C(f)\inf_{|\lambda|=1} \| (V_\phi f)^2-\lambda(V_\phi g)^2\|.$$
This gives a form of stability for STFT phase retrieval with window $\Phi$; in particular, if $|V_\Phi f|=|V_\Phi g|\in L^4$ we deduce immediately from the above quantitative estimates that this implies that $f$ is a unimodular multiple of $g$. 
\end{remark}

\subsection*{Acknowledgments}
We thank Francesca Bartolucci for interesting discussions that inspired some of the content in this article.

\bibliographystyle{plain}
\bibliography{sources}

\end{document}